\documentclass[a4paper,11pt]{article}
\usepackage{array}
\usepackage{theorem}
\usepackage{amsmath,amscd,amssymb}
\usepackage{latexsym}
\usepackage{epsfig}
\theorembodyfont{\sl}

\newtheorem{lemma}{Lemma}[section]

\newtheorem{proposition}[lemma]{Proposition}
\newtheorem{theorem}[lemma]{Theorem}
\newtheorem{corollary}[lemma]{Corollary}
\newtheorem{example}[lemma]{Example}

\newtheorem{remark}[lemma]{Remark}

\newcommand{\CC}{\mathbb C}

\newcommand{\FF}{\mathbb F}

\newcommand{\PP}{\mathbb P}
\newcommand{\QQ}{\mathbb Q}
\newcommand{\RR}{\mathbb R}

\newcommand{\ZZ}{\mathbb Z}
\newcommand{\cA}{\mathcal A}

\newcommand{\cD}{\mathcal D}
\newcommand{\cE}{\mathcal E}
\newcommand{\cF}{\mathcal F}

\newcommand{\cH}{\mathcal H}

\newcommand{\cO}{\mathcal O}

\newcommand{\cS}{\mathcal S}

\newcommand{\To}{\longrightarrow}
\newcommand{\Mapsto}{\mapstochar\longrightarrow}

\newcommand{\tensor}{\otimes}

\renewcommand{\Tilde}{\widetilde}

\renewcommand{\Bar}{\overline}

\newcommand{\cross}{\times}

\newcommand{\imic}{\cong}

\newcommand{\splitextn}{\rtimes}

\newcommand{\Eins}{{\mathbf 1}}

\newcommand{\sfrac}[2]{{\textstyle{\frac{#1}{#2}}}}

\newcommand{\SL}{\mathop{\mathrm {SL}}\nolimits}
\newcommand{\SO}{\mathop{\mathrm {SO}}\nolimits}
\newcommand{\Sp}{\mathop{\mathrm {Sp}}\nolimits}

\newcommand{\Orth}{\mathop{\null\mathrm {O}}\nolimits}

\newcommand{\rank}{\mathop{\mathrm {rank}}\nolimits}
\newcommand{\sign}{\mathop{\mathrm {sign}}\nolimits}

\newcommand{\Hom}{\mathop{\mathrm {Hom}}\nolimits}
\newcommand{\Aut}{\mathop{\mathrm {Aut}}\nolimits}

\newcommand{\id}{\mathop{\mathrm {id}}\nolimits}
\newcommand{\sn}{\mathop{\mathrm {sn}}\nolimits}

\newcommand{\ord}{\mathop{\null\mathrm {ord}}\nolimits}

\newcommand{\latt}[1]{{\langle{#1}\rangle}}
\renewcommand{\div}{\mathop{\mathrm {div}}\nolimits}
\newcommand{\Kthree}{\mathop{\mathrm {K3}}\nolimits}
\newcommand{\comm}[1]{{[{#1},{#1}]}}
\newcommand{\abel}[1]{{{#1}\null^{\mathop{\mathrm {ab}}\nolimits}}}
\newcommand{\group}[1]{{\langle{#1}\rangle}}
\newcommand{\mcomm}[1]{{[\![{#1}]\!]}}
\newcommand{\onto}{\twoheadrightarrow}
\newcommand{\qedsymbol}{\mbox{$\Box$}}
\newcommand{\qed}{\unskip\nobreak\hfil\penalty50\hskip1em\hbox{}\nobreak
\hfill\qedsymbol\parfillskip=0pt\finalhyphendemerits=0}
\newenvironment{proof}{\begin{ProofwCaption}{Proof}}{\end{ProofwCaption}}
\newenvironment{ProofwCaption}[1]
 {\addvspace\theorempreskipamount \noindent{\it #1.}\rm}
 {\qed \par \addvspace\theorempostskipamount}

\begin{document}

\title{Abelianisation of orthogonal groups and the fundamental group of modular
varieties}
\author{V.~Gritsenko, K.~Hulek and G.K.~Sankaran}
\maketitle
\begin{abstract}
\noindent We study the commutator subgroup of integral orthogonal groups
belonging to indefinite quadratic forms. We show that the index of
this commutator is~$2$ for many groups that occur in the construction
of moduli spaces in algebraic geometry, in particular the moduli of
$\Kthree$ surfaces. We give applications to modular forms and to
computing the fundamental groups of some moduli spaces.
\end{abstract}

\noindent Many moduli spaces in algebraic geometry can be described via period
domains as quotients of a symmetric space by a discrete group, or
modular group. We shall be concerned with the case of the symmetric
space $\cD_L$ associated with a lattice $L$ of signature $(2,n)$, and
discrete subgroups of the orthogonal group $\Orth(L)$ that act on
$\cD_L$. Such groups arise in the study of the moduli of $\Kthree$
surfaces and of other irreducible symplectic manifolds
(see~\cite{GHSK3}, \cite{GHSsympl} and the references there), and of
polarised abelian surfaces. Orthogonal groups of indefinite forms also
appear elsewhere in geometry, for instance in the theory of
singularities (see~\cite{Br}, \cite{Eb}). In this paper we study
the commutator subgroups and abelianisations of orthogonal modular
groups of this kind, especially for lattices of signature $(2,n)$.
\smallskip

\noindent
\emph{Notation.} For definitions and notation concerning locally
symmetric varieties and toroidal compactification we refer
to~\cite{GHSorth}.

We write $\latt{X}$ for the group generated by a subset $X$ of some
group. If $n$ is an integer $\latt{n}$ means the rank-$1$ lattice
generated by an element of square~$n$.

For a group $G$, we
write $\comm{G}$ for the commutator subgroup (derived subgroup) of $G$
(not $G'$ because we want to keep the notation $\Orth'(L)$ from
\cite{Kn1}) and we use $\abel{G}$ for the abelianisation, i.e.\ the
quotient $G/\,\comm{G}$, which is also the group $\Hom(G,\CC^\cross)$
of characters of $G$.
\medskip

The commutator subgroup and the abelianisation of a modular group
carry important information about modular forms.  For example the fact
that $\abel{\SL_2(\ZZ)}\imic \ZZ/12\ZZ$ reflects the existence of the
Dedekind $\eta$-function.  Its square $\eta(\tau)^2$ is a modular form
with respect to $\SL_2(\ZZ)$ with a character of order $12$. However,
the commutator subgroup of orthogonal modular groups is generally not
known. We are aware of two previous studies.

In \cite{GH1} two of us analysed the commutator of the
\emph{paramodular group} $\Gamma_t$ (the integral symplectic group of
a symplectic form with elementary divisors $(1,t)$), which is the
modular group of the moduli space $\cA_t$ of polarised abelian
surfaces with a polarisation of type $(1,t)$.  According to
\cite[Theorem 2.1]{GH1}
\begin{equation}\label{sp-commut}
\abel{\Gamma_t}\imic (\ZZ/(t,12)\ZZ)\times (\ZZ/(2t,12)\ZZ).
\end{equation}
We note that $\Gamma_1=\Sp_4(\ZZ)$. The projectivised paramodular
group $\Gamma_t/\{\pm\Eins\}$ is isomorphic to the stable special orthogonal group
$\Tilde\SO^+(\Lambda_{2t})$
(see ~\eqref{defineorth} below), associated with the lattice
$\Lambda_{2t}=2U\oplus\latt{-2t}$
where $U$ is the hyperbolic plane.
Therefore ~\eqref{sp-commut} is a result about orthogonal groups of signature
$(2,3)$.

S.~Kondo (\cite[Main theorem]{Ko1}) considered the lattice $L_{2d}$ of
signature $(2,19)$ associated with moduli of polarised $\Kthree$
surfaces of degree~$2d$. He proved that the abelianisation of the
modular group $\Tilde\Orth^+(L_{2d})$ is an elementary abelian
$2$-group whose order divides~$8$. We show in Theorem~\ref{com2} that
this group is in fact of order~$2$.  Moreover a similar result is true
for a large class of orthogonal groups that appear in the theory of
moduli spaces.

The paper is organised as follows. In
Section~\ref{section:commutators} we make some basic definitions,
state our main results and give some examples. We prove some of the
results straight away as corollaries of a theorem of
Kneser. Section~\ref{section:vanishing} gives an application to the
theory of modular forms, showing that in many cases of interest the
order of vanishing of a modular form at a cusp is necessarily an
integer. In Section~\ref{section:JacobiEichler} we describe  the
Eichler transvections, which are special unipotent elements of the
orthogonal groups we are interested in, and the Jacobi group, and use
them to obtain suitable generators for the modular
groups. Section~\ref{section:strongapprox} is mainly devoted to
the proof of Theorem~\ref{com2} but also includes some remarks about
the number field case. We conclude in
Section~\ref{section:fundamental} with some applications to
fundamental groups of moduli spaces. In particular we show that the
compactified moduli spaces of polarised $\Kthree$ surfaces and of
polarised abelian surfaces are simply-connected.
\bigskip

\noindent{\it Acknowledgements:\/} The first and the second author are
grateful to the Max-Planck-Institut f\"ur Mathematik in Bonn for
support and for providing excellent working conditions.

\section{Commutator subgroups}\label{section:commutators}

In this section $L$ is always an integral even lattice, i.e.\ a free
$\ZZ$-module with a non-degenerate bilinear form $(\cdot\,,\cdot)\colon
L\times L\to \ZZ $ such that $(u,u)=u^2\in 2\ZZ$ for any $u\in L$.
The dual lattice
$$
L^\vee=\{v\in L\otimes \QQ \mid (v,l)\in \ZZ\ \ \forall\, l\in L\}
$$ 
contains $L$.  We denote the \emph{discriminant group} of $L$ by
$D(L)=L^{\vee}/L$. It carries a quadratic forms with values in
$\QQ/2\ZZ$. The \emph{stable orthogonal group} $\Tilde\Orth(L)$ is
defined as the kernel
$$
\Tilde{\Orth}(L) = \ker (\Orth (L) \to \Orth(D(L)))
$$
of the natural projection to the finite orthogonal group $\Orth(D(L))$.

For an indefinite lattice there are two ways to choose
the real spinor norm because $\Orth(L,b)=\Orth(L,-b)$
where $b$ is the bilinear form on $L$.
We note that  the
different spinor norms agree on $\SO(L\otimes \QQ)$.
For any field $K$ different from $\FF_2$, any $g\in \Orth(L\otimes K)$
can be represented as the product of reflections
$g=\sigma_{v_1}\sigma_{v_2}\dots\sigma_{v_m}$, where $v_i\in L\otimes
K$.  We define the \emph{spinor norm over $K$} as follows (see
\cite{KnS}):
$$
\sn_K(g)=(-\frac{(v_1,v_1)}2)\cdot \ldots \cdot
(-\frac{(v_m,v_m)}2)(K^{\times})^2.
$$
Thus $\sn_K\colon \Orth(L\otimes K)\to K^\cross/(K^\cross)^2$ is a
group homomorphism.
We have made this choice $-(\ ,\ )$ in the definition of $\sn$
because it is convenient for the geometric applications when the
lattices have signature $(2,n)$. In that case, reflection with respect
to a vector with negative norm fixes the connected component of the
homogeneous domains.

We define three subgroups of $\Orth(L)$:
\begin{eqnarray}\label{defineorth}
\Orth^+(L)&=&\Orth (L)\cap \ker \sn_\RR,\nonumber\\
\Tilde{\Orth}^+(L)&=&\Tilde{\Orth}(L)\cap \Orth^+(L),\\
\Orth'(L)&=&\SO(L)\cap \ker \sn_\QQ.\nonumber
\end{eqnarray}
$\Orth'(L)$ is sometimes called the \emph{spinorial kernel}. We also
use the notation $\SO^+(L)=\Orth^+(L)\cap\SO(L)$ and
$\Tilde{\SO}^+(L)=\Tilde\Orth^+(L)\cap\SO(L)$; but $\Orth'(L)$ is
already a subgroup of $\SO(L)$.

If $a\in L$ and $a^2=-2$ then $a$ is called a \emph{$(-2)$-vector} or
\emph{root}. The reflection
$$
\sigma_a\colon v \To v-\tfrac{2(a,v)}{(a,a)}\,a
$$
determined by $a$ belongs to $\Tilde{\Orth}^+(L)$.

The \emph{Witt index} of $L$ over a field $K$ is the maximal dimension
of a totally isotropic subspace of $L\otimes K$.  For any prime $p$
the \emph{$p$-rank} of $L$, denoted by $\rank_p(L)$, is the maximal
rank of the sublattices $M$ in $L$ such that $\det(M)$ is coprime
to~$p$.  By the \emph{integral hyperbolic plane} we mean the lattice
$$
U:=\ZZ e\oplus \ZZ f
\qquad\text{where }\ (e,e)=(f,f)=0, \ (e,f)=1.
$$

The following result of Kneser is very important for us. It allows us
to use reflections to generate certain orthogonal groups over the
integers, not just over a field.
\begin{theorem}(\cite[Satz~4]{Kn1})\label{thKn}
Let $L$ be an even integral lattice of Witt index $\ge 2$ over
$\RR$.  We assume that $L$ represents $-2$ and that $\rank_3(L)\ge 5$
and $\rank_2(L)\ge 6$. Then $\Orth'(L)$ is generated by the products of
reflections $\sigma_a\sigma_b$ where $a,b\in L$ and $a^2=b^2=-2$.
\end{theorem}
We shall say that a lattice \emph{satisfies the Kneser conditions} if
it satisfies the conditions of Theorem~\ref{thKn}.
\begin{corollary}\label{comparespinors}
If $L$ satisfies the Kneser conditions, then
$$
\Orth'(L)=\Tilde\SO^+(L).
$$
\end{corollary}
\begin{proof}
According to Theorem~\ref{thKn}, $\Orth'(L)$ is a subgroup of
$\Tilde\SO^{+}(L)$.  But by \cite[Satz~2]{Kn1}, the local orthogonal
group $\Tilde\Orth(L\otimes \ZZ_p)$ is generated by reflections with
respect to $(-2)$-vectors for every finite prime~$p$. Therefore for
any $g\in \Tilde\SO(L)$ we have
$\sn_{\QQ_v}(g)=1\in{\QQ_v^\cross}/(\QQ_v^\cross)^2$ for every
$p$-adic valuation~$v$ on~$\QQ$.
Therefore for any $g\in \Tilde\SO^+(L)$
$\sn_{\QQ}(g)=\sn_\RR(g)\cdot\prod_p\sn_{\QQ_p}(g)=1$.
\end{proof}

Note that Corollary~\ref{comparespinors} is also true with opposite
signs, i.e.\ if $L$ contains at least one $2$-vector and we consider
the reflections $\sigma_a$ with $a^2=2$.  To see this, simply multiply
the quadratic form of the lattice $L$ by $-1$.
\smallskip

Our first result on the commutator is a corollary of Kneser's theorem.
We consider the group $\Tilde\Orth^+(L)$, which is the main group in
the geometric applications we shall give later.

\begin{theorem}\label{com1}
Let $L$ be a lattice which satisfies the Kneser conditions. Then
$\abel{\Tilde\Orth^+(L)}$ (resp. $\abel{\Tilde\SO^+(L)}$)
is an abelian $2$-group. Its order divides $2^N$ (resp. $2^{N-1}$),
where $N$ is the number of different $\Tilde\Orth^+(L)$-orbits
(resp. $\Tilde\SO^+(L)$-orbits) of $(-2)$-vectors
in~$L$.
\end{theorem}
\begin{proof}
For roots $a,b\in L$ we write $a\equiv b\mod \Tilde\Orth^+(L)$ if there
exists $g\in \Tilde\Orth^+(L)$ such that $g(a)=b$.  In this case
$\sigma_{g(a)}=g\sigma_ag^{-1}$ and $\sigma_a\sigma_b\in
\comm{\Tilde\Orth^+(L)}$.  By Theorem~\ref{thKn}, any element of
$\Tilde\SO^+(L)$ is a product of reflections by $(-2)$-vectors, and
since $L$ represents $-2$ the same is true for
$\Tilde\Orth^+(L)$. Using this, and the evident property
$\sigma_u\sigma_v=\sigma_{\sigma_u(v)}\sigma_u$, we can rewrite any
class modulo commutator as the class of a product $\sigma_{a_1}\dots
\sigma_{a_n}$, where the $(-2)$-vectors $a_i$ all belong to different
$\Tilde\Orth^+(L)$-orbits. The square of such a class can be written
as the class of a product of elements $\sigma_{b_i}\sigma_{c_i}$ where
$b_i\equiv c_i\mod \Tilde\Orth^+(L)$, so it belongs to the
commutator. Exactly the same argument works for $\Tilde\SO^+(L)$.
\end{proof}

Let us remark that if $L=2U\oplus L_0$ and $L_0$ contains a sublattice
isomorphic to $A_2$ then $L$ satisfies the Kneser conditions.

We do not know exactly how far the conditions on $\rank_{2}(L)$ and
$\rank_{3}(L)$ are necessary in Theorem~\ref{com1}. For
$\Tilde\SO^+(L)$ they cannot be weakened much, as the following examples show.

\begin{example}\label{example:2rank}
Take $t\not\equiv 0\bmod 3$ and take $L= \Lambda_{2t}=2U\oplus
\latt{-2t}$. Then $\Lambda_{2t}$ satisfies all the Kneser conditions
except that $\rank_2(\Lambda_{2t})=4$. In this case
$\abel{\Tilde\SO^+(\Lambda_{2t})}$ contains a subgroup isomorphic to
$\ZZ/4\ZZ$ if $t$ is even.
\end{example}

In this case $\Tilde\SO^+(\Lambda_{2t})$ is isomorphic to the projective
paramodular group $\Gamma_t/\{\pm \Eins_4\}$ (see \cite{GH2}). Hence
the $4$-torsion element appears because of
equation~\eqref{sp-commut}. If $3|t$, then there is also
$3$-torsion (and the Kneser condition on $\rank_3(\Lambda_{2t})$ fails also).

However, we do not know an example where the conclusion of
Theorem~\ref{com1} fails and the Kneser conditions fail only because
$\rank_2(L)=5$.
\begin{example}\label{example:3rank}
Take $L=2U\oplus A_2(-3)$. Then $L$ satisfies all the Kneser
conditions except that $\rank_3(L)=4$. In this case
$\abel{\Tilde\SO^+(L)}$ contains a subgroup isomorphic to $\ZZ/3\ZZ$.
\end{example}
In fact, Desreumaux has constructed a modular form with respect to
$\Tilde\SO^+(L)$ with a character of order~$3$ (see~\cite{De}).

\begin{proposition}\label{unimodular}
Let $L=2U\oplus L_1$ be an even unimodular lattice of rank at least
$6$. Then $\abel{\Tilde{\SO}^+(L)}$ is trivial and
$\abel{\Tilde\Orth^+(L)}\imic \ZZ/2\ZZ$.
\end{proposition}
\begin{proof}
  $L$ satisfies the Kneser conditions so $\Tilde\SO^+(L)=\SO^+(L)$ is
  generated by products $\sigma_a\sigma_b$ with $a^2=b^2=-2$. The
  orbit of a $(-2)$-vector $a$ is determined by its image in the
  discriminant group (this is a case of the Eichler criterion,
  from~\cite[\S 10]{Ei}: see Proposition~\ref{ET}(i), below).  But that
  group is trivial. Therefore there exists $g\in \SO^+(L)$ such that
  $g(a)=b$. But then
  $\sigma_a\sigma_b=\sigma_a\sigma_{g(a)}=\sigma_ag\sigma_ag^{-1}$ is
  a commutator.
\end{proof}

The lattices that appear in the theory of the moduli spaces of
symplectic varieties frequently contain two integral hyperbolic planes
even if they are not unimodular.  In the case of polarised $\Kthree$
surfaces of degree $2d$ the lattice
\begin{equation}\label{L2d}
L_{2d}=2U\oplus 2E_8(-1)\oplus \latt{-2d}
\end{equation}
occurs.

In the main theorem of \cite{Ko1} it was proved that the order
of $\abel{\Tilde\Orth(L_{2d})}$ divides~$16$ or equivalently that
the order of
$\abel{\Tilde\Orth^+(L_{2d})}$ divides~$8$.
But  there are at most  two $\Tilde\Orth^+(L_{2d})$-orbits of
$(-2)$-vectors in $L_{2d}$ \cite[Proposition~2.4(ii)]{GHSorth}.
Hence by Theorem~\ref{com1} the order of $\abel{\Tilde\Orth^+(L_{2d})}$
divides~$4$. (There are two orbits  if and only if  $d\not\equiv 1$
mod~$4$.)

But in fact the order is~$2$, for any~$d$. This is a
special case of the following, which is our main theorem in this
paper.

\begin{theorem}\label{com2}
Let $L$ be an even integral lattice containing at least
two hyperbolic planes, such that $\rank_2(L)\ge 6$ and
$\rank_3(L)\ge 5$.  Then $\abel{\Tilde{\SO}^+(L)}$ is trivial and
$\abel{\Tilde\Orth^+(L)}\imic \ZZ/2\ZZ$.
\end{theorem}
The proof of Theorem~\ref{com2} will be given in
Subsection~\ref{subsect:proof} below. The main tool is the
Siegel--Eichler orthogonal transvections introduced in
\cite[Ch. 1--2]{Ei}.

For $L$ of signature $(2,n)$ Borcherds proposed in~\cite{Bo} a very
powerful construction of automorphic forms with respect to subgroups
of $\Orth^+(L)$.  We can use Theorem~\ref{com2} to give an answer to
the question discussed in the remark on page~546 of~\cite{Bo}.

\begin{corollary}\label{ch-det}
For $L$ a lattice as in Theorem~\ref{com2}, the orthogonal group
$\Tilde{\Orth}^+(L)$ has only one non-trivial character,
namely~$\det$, and $\Tilde\SO^+(L)$ has no non-trivial characters.
\end{corollary}

\begin{remark}
In many cases where the quotient $\Tilde{\Orth}^+(L) \backslash \cD_L$
of the homogeneous domain $\cD_L$ associated to $L$ (see Section
{\ref{section:vanishing}} below) represents a moduli functor. In
these cases, Corollary~\ref{ch-det} also means that the torsion group
of the Picard group of the associated moduli stack is $\ZZ/2\ZZ$. See also
\cite[Proposition 2.3]{GH1} for the case of abelian surfaces.
\end{remark}

Returning to the polarised $\Kthree$ lattice $L_{2d}$
we note that $\Orth({L^\vee}_{2d}/L_{2d})\cong (\ZZ/2\ZZ)^{\rho(d)}$
where $\rho(d)$ is the number of divisors of $d$ (see \cite{GH2}).
Then according to Theorem \ref{com2}
$[\Orth(L_{2d}), \Orth(L_{2d})]=\Tilde{\SO}^+(L_{2d})$ and
$$
\abel{\Orth(L_{2d})}\cong (\ZZ/2\ZZ)^{\rho(d)+2}.
$$

\section{Vanishing order of cusp forms}\label{section:vanishing}

The modular form $\eta^2$ is a cusp form for $\SL_2(\ZZ)$, but it has
highly non-trivial character and its order of vanishing at a cusp is
not an integer (it is not a section of a line bundle, only of a
$\QQ$-line bundle). In~\cite{GH1} there are also many examples of
modular forms with more complicated characters for orthogonal groups
of lattices of signature~$(2,3)$. On the other hand,
Corollary~\ref{ch-det} shows that for lattices satisfying the
conditions of Theorem~\ref{com2} there are no modular forms with
complicated characters (indeed no complicated characters). In this
section we consider lattices of signature $(2,n)$ and analyse the
relation between the character of a modular form and its possible
orders of vanishing at cusps. We use the following notation: $\cD_L$
is the symmetric domain associated with the lattice $L$; $\cD^\bullet$
is the affine cone on $\cD$; $F$ is a cusp, corresponding to an
isotropic subspace defined over $\QQ$, and $\cD(F)$ a suitable
neighbourhood of it; $U(F)$ is the centre of the unipotent radical of
the stabiliser of $F$ in $\Aut(\cD)$ and $U(F)_\ZZ$ is the
intersection of $U(F)$ with the modular group; $\cH_n$ is a tube
domain. For more details we refer to~\cite{GHSorth} and for the
general theory of toroidal compactification to~\cite{AMRT}.

\begin{proposition}\label{integer_vanishing}
  Let $L=2U\oplus L_0$ be a lattice of signature $(2,n)$ containing
  two hyperbolic planes and let $\psi$ be a modular form with
  character $\det$ or trivial character for an arithmetic subgroup of
  $\Tilde\Orth^+(L)$. Then the order of vanishing of $\psi$ along any
  boundary component $F$ of $\cD_L$ is an integer.
\end{proposition}

\begin{proof}
If $\psi$ is of weight $k$ then near the boundary component $F$ we have
$$
\psi(gZ)=j(g,Z)\chi(g)\psi(Z),
$$
where $Z\in \cD_L(F)$ and $g\in U(F)_\ZZ$, for some factor of
automorphy $j$ and $\chi$ the character of the modular form $\psi$. If
the factor $j(g,Z)\chi(g)$ is equal to $1$ for every $g\in U(F)_\ZZ$
then $\psi$ is a section of a line bundle near $F$ and its order of
vanishing along $F$ is therefore an integer.

Under the hypotheses of the proposition, we do indeed have $\chi(g)=1$
because $g$ is unipotent and therefore has trivial determinant. It
therefore remains to check that the factor of automorphy $j(g,Z)$ is
also trivial for $g\in U(F)_\ZZ$.

If $F$ is of dimension $1$ then according to~\cite[Lemma 2.25]{GHSK3}
we have
$$
U(F)=\left\{\begin{pmatrix}I&0&\begin{pmatrix}0&ex\\ -x&0\end{pmatrix}\\
0&I&0\\ 0&0&I\end{pmatrix}\mid x\in\RR\right\}.
$$
But the automorphy factor is given by the last ($(n+2)$-nd) coordinate of
$g\big(p(Z)\big)\in\cD_L$, where
\begin{eqnarray*}
p\colon \cH_{n}&\To& \cD_L\\
Z=(z_n, \dots, z_1)
&\Mapsto& \big( -\frac{1}{2}(Z, Z)_{L_1}:
z_n: \cdots :z_1 : 1\big)
\end{eqnarray*}
is the tube domain realisation of $\cD_L$: see
\cite[Section~3]{GHSorth} or \cite[Section 2]{Gr2} for the notation and
more detail.  From this
description it is immediate that $j(g,Z)=1$ for $g\in U(F)_\ZZ$.

If $F$ is of dimension $0$ then $F$ corresponds to some isotropic
vector $v\in L$, and $U(F)$ is the centre of the unipotent radical
of the stabiliser of $v$. In this case the unipotent radical is
abelian. With respect to a basis of $L\tensor \QQ$ in which $v$ is
the last ($n+2$-nd) element, the penultimate ($n+1$-st) element $w$
is also isotropic and the remaining elements span the orthogonal
complement $L'$ of those two, we have
$$
U(F)=\left\{\begin{pmatrix}I_n&b&0\\
0&1&0\\ {}^tc&x&1\end{pmatrix}\mid L'b+\alpha c=0,\ {}^tb L'
b+2\alpha x=0\right\}.
$$
Here $b$ and $c$ are column vectors, $x\in \RR$ and
$\alpha=(w,v)_L$: compare \cite[(2.7)]{Ko2}. In this case the tube
domain is contained in $\CC^n$ and is identified with a subset of the
locus $z_{n+1}=1\subset \cD^\bullet_L$. The automorphy factor $j(g,Z)$
is therefore equal to the $n+1$-st coordinate of $g(p(Z))$, where
$p(Z)_{n+1}=1$; but this is~$1$ as $p(Z)$ is a column vector.
\end{proof}
From Proposition~\ref{integer_vanishing} and Corollary~\ref{ch-det} we
have immediately the following result.
\begin{corollary}\label{order_1}
If $L$ is a lattice of signature $(2,n)$ satisfying the conditions of
Theorem~\ref{com2}, then any cusp form for $\Tilde\Orth^+(L)$ or
$\Tilde\SO^+(L)$ vanishes to integral order along any toroidal
boundary divisor. In particular the order of vanishing of a cusp form
along a boundary divisor is always at least~$1$.
\end{corollary}

\section{Eichler transvections and the Jacobi group}\label{section:JacobiEichler}

In this section we analyse the modular groups and construct useful
sets of generators for them.

\subsection{Eichler transvections}\label{subsect:transvections}

Let $V=L\otimes \QQ$ be a quadratic space over $\QQ$ and let $e$ be an
isotropic vector in $V$ (i.e.\ $e^2=0$) and $a\in e^\perp_V$. The map
$$
t'(e,a)\colon v \Mapsto v-(a,v)e \qquad (v\in e_V^\perp)
$$
belongs to the orthogonal group $\Orth(e_V^\perp)$.

\begin{lemma}\label{lemma:unique_extension}
$t'(e,a)$ extends to a unique element $t(e,a)\in \Orth(V)$.
\end{lemma}
\begin{proof}
We first complete $e$ to a rational hyperbolic plane $\QQ e\oplus \QQ
f \subseteq V$.  If there exist $\gamma_1$, $\gamma_2\in \Orth(V)$
such that $\gamma_1(e)=\gamma_2(e)=e$ and $\gamma_1|_{e_V^\perp}=
\gamma_2|_{e_V^\perp}$, then they take the same value on $f$.  The
unique orthogonal extension of $t'(e,a)$ on $V$ is given by the map
\begin{equation}\label{t2}
t(e,a)\colon v\Mapsto v-(a,v)e+(e,v)a-\frac{1}{2}(a,a)(e,v)e.
\end{equation}
This element is called an \emph{Eichler transvection} (see \cite[\S 3]{Ei}).
\end{proof}
We note that $t(e,a)$ acts as the identity on $e_V^\perp\cap a_V^\perp
\subset V$. In particular $t(e,a)(e)=e$.
Using Lemma \ref{lemma:unique_extension} it is  easy to see that
\begin{equation}\label{t3}
t(e,a)t(e,b)=t(e,a+b)\quad\text{and}\quad
t(e,a)^{-1}=t(e,-a),
\end{equation}
\begin{equation}\label{t4}
\gamma t(e,a)\gamma^{-1}=t(\gamma(e),\gamma(b))
\qquad \forall\, \gamma\in \Orth(V),
\end{equation}
\begin{equation}\label{t5}
t(xe,a)=t(e,xa),\quad t(e,xe)=\id \qquad \forall\, x\in \QQ^*,
\end{equation}
\begin{equation}\label{t6}
t(e,a)=\sigma_a\sigma_{a+\frac{1}2(a,a)e}\qquad \text{if }(a,a)\ne 0.
\end{equation}
Using equation~\eqref{t6} one can prove (see \cite[(3.12)]{Ei}) that
for any non-isotropic $a$ orthogonal to the rational hyperbolic plane
$\QQ e\oplus\QQ f$
\begin{equation}\label{t7}
t(f,a)\,t(e,\, \tfrac{2}{(a,a)}a)\,t(f,a)
=\sigma_a\sigma_{e+(2/(a,a))f}.
\end{equation}
From the definition~\eqref{t2} we see that any transvection $t(e,a)$
is unipotent.  From equation~\eqref{t6} we have that $t(e,a)\in
\SO^+(L\otimes \QQ)$.  According to equation~\eqref{t2}
\begin{equation}\label{t8}
t(e,a)\in \Tilde\SO^+(L)\quad \text{for any }\,e\in L,\ a\in L
\ \text{with } (e,e)=(e,a)=0.
\end{equation}
Moreover for any primitive isotropic $e$ in $L$
$$
t(e,*)\colon e^\perp_L\To \Tilde\SO^+(L)
$$
is a homomorphism of groups with kernel $\ZZ e$.

One can also give a description of the transvections in the terms of
the Clifford algebra of $L$.  For any isotropic $e\in L\otimes \QQ$
and any $a$ such that $(a,e)=0$ we have that $1-ea\in
\hbox{Spin}(L\otimes \QQ)$ and $\pi(1-ea)=t(e,a)$, where
$\pi(\gamma)(v)=\gamma v\gamma^{-1}$ for any $\gamma$ in the Clifford
group (see, e.g., \cite{HOM}).

\subsection{The Jacobi group}\label{subsect:jacobi_group}

Suppose $L=U\oplus U_1\oplus L_0$, where $U=\ZZ e\oplus\ZZ f$ and
$U_1=\ZZ e_1\oplus\ZZ f_1$ are two integral hyperbolic planes. Let $F$
be the totally isotropic plane spanned by $f$ and $f_1$ and let $P_F$
be the parabolic subgroup of $\SO^+(L)$ that preserves~$F$. This
corresponds to a $1$-dimensional cusp of the modular variety
$\SO^+(L)\backslash\cD_L$.  We choose a basis of $L$ of the form
$(e,e_1,\dots , f_1,f)$.  The subgroup $\Gamma^J(L_0)$ of $P_F$ of
elements acting trivially on the sublattice $L_0$ is called the
\emph{Jacobi group}.

The Jacobi group is isomorphic to the semidirect product of
$\SL_2(\ZZ)$ with the Heisenberg group $H(L_0)$, the central extension
$\ZZ\splitextn (L_0\times L_0)$. More precisely (see~\cite{Gr2} for more
information) we define elements $[A]\in \Gamma^J(L_0)$ for
$A\in\SL_2(\ZZ)$ and $[u,v;z]\in \Gamma^J(L_0)$ for $u$, $v\in L_0$,
$z\in \ZZ$ by
$$
[A]:=\begin{pmatrix}
A^*&0&0\\
0&\Eins_{n_0}&0\\
0&0&A
\end{pmatrix},
$$
$$
[u,v;z]:=
\begin{pmatrix}
1&0&-{}^tvS_0&-(u,v)-z& -(v,v)/2\\
0&1&-{}^tuS_0&-(u,u)/2&z\\
0&0&\Eins_{n_0}&u&v\\
0&0&0&1&0\\
0&0&0&0&1\end{pmatrix},
$$
where $S_0$ is the matrix of the quadratic form $L_0$ of rank
$n_0$, we consider $u$ and $v$ as column vectors, and
$A^*=\left(\smallmatrix
0&1\\1&0\endsmallmatrix\right)A^{-1}\left(\smallmatrix
0&1\\1&0\endsmallmatrix\right)$. Thus any element of $\Gamma^J(L_0)$
may be written in the form $[A]\cdot[u,v;z]$ for suitable $A$, $u$,
$v$ and $z$.

The Jacobi group is generated by the transvections
\begin{equation}\label{tSL}
t(e,f_1)=\left[\left(\smallmatrix 1&1\\0&1\endsmallmatrix\right)\right],\quad
\ t(f,e_1)=\left[\left(\smallmatrix 1&0\\-1&1\endsmallmatrix\right)\right],
\end{equation}
\begin{equation}\label{tHL}
t(e,v)=[0,v;0],\quad t(e_1,u)=[u,0;0], \quad t(e,e_1)=[0,0;1].
\end{equation}
Note that $t(e,e_1)$ generates the centre of the Heisenberg group.  It
is easy to see, using only the elementary divisor theorem, that
$$
\SL_2(\ZZ)\times \SL_2(\ZZ)/\{\pm(\Eins_2,\Eins_2)\}\imic\SO^+(U\oplus
U_1).
$$
If we identify $xe+x_1e_1+y_1f_1+yf\in U\oplus U_1$ with
$X=\left(\smallmatrix x_1&x\\y&-y_1\endsmallmatrix\right)$, the
isomorphism is given by
\begin{equation}\label{isom(2,2)}
(B,A) \Mapsto
\left(X\mapsto BXA^{-1}\right).
\end{equation}
The map $X\to BXA^{-1}$ certainly preserves the quadratic form $-\det
X$.  Its kernel is the centre $\pm(\Eins_2,\Eins_2)$ of
$\SL_2(\ZZ)\times \SL_2(\ZZ)$.  The first copy of $\SL_2(\ZZ)$
(parametrised by $B$ in~\eqref{isom(2,2)}) is generated by the
transvections $t(e,e_1)$ and $t(f,f_1)$, the second by $t(e,f_1)$ and
$t(f,e_1)$.  From the representation above and from the elementary
divisor theorem for $2\times 2$ matrices there follows the next lemma,
which is well-known.
\begin{lemma}\label{SO(2,2)}
$\SO^+(U\oplus U_1)$ is generated by the four  transvections
$t(e,e_1)$, $t(e,f_1)$, $t(f,e_1)$ and $t(f,f_1)$.
For any $v\in U\oplus U_1$ there exists $g\in \SO^+(U\oplus U_1)$
such that $g(v)\in U_1$.
\end{lemma}

\subsection{The group $E(L)$ of unimodular transvections}\label{subsect:unimodular}

The \emph{divisor} $\div(l)$ of $l\in L$ is the positive generator of
the ideal $(l,L)\subset \ZZ$, so $l^*=l/\div(l)$ is a primitive
element of the dual lattice $L^\vee$. Therefore $l^*\pmod L$ is an
element of order $\div(l)$ of the discriminant group $D(L)$ and $\div(l)$ is
a divisor of $\ord(D(L))=|\det (L)|$.  One can complete an isotropic element
$e\in L$ to an integral isotropic plane $U=\ZZ e\oplus\ZZ f\subset L$
if and only if $\div(e)=1$.  We call such an isotropic vector
\emph{unimodular}.  For a unimodular isotropic vector $e$ we have
$L=U\oplus L_1$.

We define $E(L)$ to be the group generated by all transvections by
unimodular isotropic vectors:
$$
E(L):=\group{\{t(e,a)\mid e,a\in L,\ (e,e)=(e,a)=0,\ \div(e)=1\}}.
$$
We have seen that $E(L)$ is a subgroup of $\Tilde{\SO}^+(L)$.  Now let
us fix a unimodular isotropic vector $e\in L$ and the decomposition
$L=U\oplus L_1$ where $U=\ZZ e\oplus\ZZ f$.  Then we set
$$
E_U(L_1):=\group{\{t(e,a),\ t(f,a)\mid a\in L_1 \}}.
$$

\begin{proposition}\label{ET}
Let $L=U\oplus U_1\oplus L_0$, where $U=\ZZ e\oplus\ZZ f$, $U_1$ is
the second copy of the integral hyperbolic plane in $L$ and
$L_1=U_1\oplus L_0$.
\begin{itemize}
\item[{\rm (i)}] If $u, v\in L$ are primitive, $(u,u)=(v,v)$ and
  $u^*\equiv v^*\mod L$, then there exists $\tau\in E_{U}(L_1)$ such that
  $\tau(u)=v$.
\item[{\rm (ii)}] $E(L)=E_{U}(L_1)$.
\item[{\rm (iii)}] $\Orth(L)=\group{E_{U}(L_1),\ \Orth(L_1)}$.
\item[{\rm (iv)}] For any $(-2)$-vector $r\in L$ there exists $\rho\in
  E_{U}(L_1)$ such that $\sigma_r= \rho\cdot\sigma_{e-f}$.
\end{itemize}
\end{proposition}
\begin{proof} (i) First we note that $\div(u)=\ord_{D(L)}(u^*)$.
Therefore $\div(u)=\div(v)=d$.  According to Lemma~\ref{SO(2,2)} there
exists $\tau_1 \in E_{U}(U_1)$ such that $\tau_1(u)\in L_1$.  Thus we may
assume that $u$ and $v$ are in $L_1$.  Then we can realise the
translation by $w=(u-v)/d$ in the sublattice $L_1$ orthogonal to $U$
as a composition of Eichler transvections:
$$
u\ \stackrel{t(e,u')}\Mapsto\
(u-de)\  \stackrel{t(f,w)}\Mapsto\
(v-de)\  \stackrel{t(e,-v')}\Mapsto\  v,
$$
where $u',v'\in L_1$ are such that $(u,u')=(v,v')=d$.
\smallskip

(ii) Let $t(u,a)$ be an arbitrary unimodular transvection in $E(L)$
with  $(u,u)=0$ and $\div(u)=1$.
According to (i) there exists $\tau \in E_{U}(L_1)$  such that
$\tau (u)=e$. By equation~\eqref{t4} we obtain that
$\tau\, t(u,a)\tau^{-1}=t(\tau(u), \tau(a))=t(e,\tau(a))$ is in $E_{U}(L_1)$.
\smallskip

(iii) Let $g\in \Orth(L)$. According to (i) and (ii) there exists
$\tau\in E_{U}(L_1)$ such that $\tau (g(e))=e$.  We have $\bigl(\tau
g(e), \tau g(f)\bigr)=(e,\tau g(f))=(e,f)=1$. Therefore
$$
(\tau g)(f)=f+b-\frac{1}2(b,b)e=t(e,b)(f),
$$
where $b\in L_1$.  Now we see that $h=t(e,-b)\tau g$ acts trivially on
$U$.  Therefore $h\in \Orth(L_1)$.
\smallskip

(iv) There exists $\tau\in E_{U}(L_1)$ such that $\tau(r)=a\in L_1$.
According to equation~\eqref{t7}
$$
\tau \sigma_r \tau^{-1}=\sigma_a=
t(f,a)\,t(e,-a)\,t(f,a)\,\sigma_{e-f}
$$
($\sigma_a$ and $\sigma_{e-f}$ commute).
To finish we use that $\sigma_{e-f}\tau\sigma_{e-f}\in E_{U}(L_1)$
for any $\tau\in E_{U}(L_1)$.
\end{proof}

Notice that (iii) is true for all the groups we have considered: for
instance, $\Tilde\Orth^+(L)=\group{E_{U}(L_1),\ \Tilde\Orth^+(L_1)}$
and similarly for $\SO$, $\Tilde\SO$, etc.. This is because in the
proof of (iii) the product $t(e,-b)\tau\in\Tilde\SO^+(L)$, which is a
subgroup of all of these groups.

All the results of Proposition~\ref{ET} are essentially to be found
in~\cite{Ei}. (i),~which is sometimes called the \emph{Eichler
  criterion}, is proved in \cite[Satz 10.4]{Ei} for lattices over
local rings. See also the second proof given in ``Anmerkungen zum
zweiten Kapitel" \cite[p.~231]{Ei}. There is a global variant in
\cite[p.85]{Br}. (iii) was proved in \cite[5.2]{Wa} for unimodular
lattices (see also \cite{P-SS}, \cite{Eb}, \cite{Gr2}).  One can prove
(ii), under an additional condition on $\rank_\pi(L)$ for all primes
ideals $\pi$, over any commutative ring, but the proof is much longer:
see \cite[Theorem 3.3(a)]{Va1}).

Proposition~\ref{ET} gives us the following result about generators of
the orthogonal group which was briefly indicated in~\cite[p.1194]{Gr2}.

\begin{proposition}\label{generators}
Let $L=U\oplus U_1\oplus L_0$ be an even lattice with two hyperbolic
planes, such that $\rank_3(L)\ge 5$ and $\rank_2(L)\ge 6$. Then
\begin{equation}\label{SOandE}
\Tilde\SO^{+}(L)=\Orth'(L)=E(L)=E_{U}(L_1)
\end{equation}
and
\begin{equation}\label{otildeplus}
\Tilde\Orth^+(L)=\group{\Gamma^J(L_0),\ \sigma_1},
\end{equation}
where $L_1=U_1\oplus L_0$, $\Gamma^J(L_0)$ is the Jacobi group,
$U_1=\ZZ e_1\oplus \ZZ f_1$ and $\sigma_1=\sigma_{e_1-f_1}$.
\end{proposition}
\begin{proof}
According to Proposition~\ref{ET}(iv) the product $\sigma_a\sigma_b$
of any two reflections with $(a,a)=(b,b)=-2$ belongs to $E(L)$.
Therefore from Theorem~\ref{thKn} and Proposition~\ref{ET}(ii) it
follows that
\begin{equation}\label{SO}
\Tilde\SO^{+}(L)
=\Orth'(L)=E(L)=E_U(L_1)=\group{\{t(c,a)\mid a\in L_1,\ c=e\text{ or }f\}}.
\end{equation}
The Jacobi group $\Gamma^J(L_0)$ contains the transvections $t(e,v)$
($v\in L_1$) and $t(f,e_1)$ (see \eqref{tSL}--\eqref{tHL}). To have
the whole group $\Tilde\SO^{+}(L)=E(L)$ we have to add $t(f,u+xf_1)$
with $u\in L_0$ and $x\in \ZZ$.  The $\SL_2(\ZZ)$-subgroup of the
Jacobi group is generated by $t(e,f_1)$ and $t(f,e_1)$. Consider the element
$S=\left[\left(\smallmatrix 0&-1\\1&0\endsmallmatrix\right)\right]
\in \Gamma^J(L_0)$. We have
$$
S(e)=-e_1,\quad  S(f)=-f_1,\quad S^2=-\id.
$$
Using equation~\eqref{t4} we deduce
\begin{eqnarray*}
\sigma_1t(f,e_1)\sigma_1&=&t(f,f_1),\\
(S\sigma_1S\sigma_1)t(e,u)(S\sigma_1S\sigma_1)^{-1}&=&t(f,u)
\quad\text{for all }u\in L_0.
\end{eqnarray*}
Therefore $\group{\Gamma^J(L_0),\ \sigma_1}$ contains all the
generators of $E_U(L_1)$.  The proposition follows from equation~\eqref{SO}.
\end{proof}

\section{Strong Approximation}\label{section:strongapprox}

In this section we prove Theorem~\ref{com2}, and make some remarks
about similar results over number fields.

It is enough to prove Theorem~\ref{com2} for $\Tilde\SO^{+}(L)$ (or,
equivalently by equation~\eqref{SO}, for $E(L)$), because
$\Tilde\Orth^{+}(L)=\group{\Tilde\SO^{+}(L),\ \sigma_{e-f}}$. Vaserstein
\cite[Theorem~3(c)]{Va1} did this under the extra assumption that that
$\rank_p(L)\ge 5$ for any odd prime~$p$.

Our method is different and Theorem~\ref{com2} does not have
the infinite set of conditions $\rank_p(L)\ge 5$ for $p>3$.
We use the strong approximation theorem
($L$ is indefinite) and the positive solution of the principal
congruence problem for the spinorial kernel
$\Orth'(L)=\Tilde\SO^{+}(L)$ for a lattice $L$ with real Witt index
$\ge 2$ (see \cite[11.4]{Kn2}).

\subsection{Proof of Theorem~\ref{com2}}\label{subsect:proof}

First we note that $\comm{E(L)}$ is an infinite normal subgroup of
$\Orth'(L)$ which is not a subgroup of its centre.  Therefore
$\comm{E(L)}$ contains a congruence subgroup of $\Orth'(L)$ of some
level~$m$. We may assume that $6$ divides~$m$. According to
Proposition~\ref{ET}(ii), the group $E(L)$ is generated by all
$t(e,u)$ and $t(f,v)$ where $u$, $v\in L_1$.  We prove that these
generators are the products of commutators in $E(L_p)$, where
$L_p=L\otimes \ZZ_p$, for any prime divisor $p$ of $m$.  For this
purpose we introduce the Eichler orthogonal transformation $P(s)\in
\SO(L\otimes \QQ_p)$ for $s\in \QQ_p^{\times}$:
$$
P(s)\colon e\Mapsto s^{-1}e,\ f\Mapsto sf,\ u\Mapsto u\ \ \forall\ u\in L_1.
$$
We have $P(s)^{-1}=P(s^{-1})$. We can describe $P(s)$ in terms of
reflections because $\sigma_{e-sf}=P(s^{-1})\psi$, where $\psi\in
\Orth(L)$ is the permutation of $e$ and $f$. Thus
$P(s)=\sigma_{e-f}\sigma_{e-sf}$.  The following formula (see
\cite[(3.16)]{Ei}) can be obtained as a corollary of \eqref{t8}:
\begin{equation}
t(f,sw)t(e,w)=t(e,(1-s\tfrac{w^2}{2})^{-1}w)\,t(f,s(1-s\tfrac{w^2}{2})w)
P((1-s\tfrac{w^2}{2})^{2})
\end{equation}
for any $w\in L_1\otimes \QQ_p$ and $s\in \QQ_p$ such that
$1-s\tfrac{w^2}{2}\ne 0$.  In particular for any $v_6\in L_1\otimes
\ZZ_p$ such that $(v_6,v_6)=6$ and $s=1$ we obtain that $P(4)$ is a
commutator in $E(L_p)$ if $p\ne 2$:
\begin{equation}
P(4)=t(f,2v_6)t(e,2^{-1}v_6)t(f,v_6)t(e, v_6).
\end{equation}
It follows that $t(e,u)$ and $t(f,v)$ are commutators in $E(L_p)$
if $p\ne 2$ or  $3$:
\begin{equation}
t(e,u)=P(4)^{-1}t(e, 3^{-1}u)P(4)t(e,-3^{-1}u).
\end{equation}

Now we consider $p=2$. Let $L=U\oplus U_1\oplus L_0$, with $U=\ZZ
e\oplus\ZZ f$ and $U_1=\ZZ e_1\oplus\ZZ f_1$. For any $u$ orthogonal
to $e$ and $f_1$ we have
$$
t(e_1,u)(e)=e,\qquad t(e_1,u)(f_1)=f_1+u-\frac{1}2(u,u)e_1.
$$
Therefore for any $s\in \QQ_2^{\times}$ we have
$$
[t(e,-sf_1),\,t(e_1,u)]=t(e,su-s\sfrac{u^2}{2}e_1).
$$
Using the same formula for $v$ and $-(u+v)$ we obtain the following
representation
\begin{multline*}
t(e, s(u,v)e_1)=\\
[t(e,-sf_1),\,t(e_1,u)]\cdot[t(e,-sf_1),\,t(e_1,v)]\cdot[t(e,-sf_1),\,t(e_1,-u-v)].
\end{multline*}
Since $\rank_2(L)\ge 6$, we can find $u,v\in L_0\otimes \ZZ_2$ such
that $(u,v)\in \ZZ_2^\times$.  Therefore taking $s=(u,v)^{-1}$ we
obtain $t(e, e_1)$ as the product of three commutators in $E(L\otimes
\ZZ_2)$.  The same argument works for $t(e, f_1)$.  Then we can
replace $e_1$ by any unimodular isotropic vector of the form
$e'_1=e_1+w-\tfrac{w^2}2f_1$ where $w\in L_0$. We note that
$(e'_1,f_1)=1$.  We can repeat the arguments above for this new
hyperbolic plane $U'_1=\langle e'_1,f_1\rangle$ and we obtain that
$t(e, e_1+w-\tfrac{w^2}2f_1)$ belongs to the commutator subgroup of
$E(L\otimes \ZZ_2)$.  Using $t(e,e_1)$, $t(e,f_1)$ and $t(e,
e_1+w-\tfrac{w^2}2f_1)$, we see that $t(e,l)$ for any $l\in L_1$ is a
commutator in $E(L\otimes \ZZ_2)$.

For $p=3$ we can use the same calculation with a vector $u\in L_0$
such that $(u,u)\in \ZZ_3^{\times}$ ($\rank_3 (L)\ge 5$).

We have proved that the generators $t(e,u)$ and $t(f,v)$ ($u,v\in L_1$)
are elements of the commutator subgroup of $E(L_p)$ for any prime
divisor $p$ of the level $m$. So we can write
$$
t(e,u)=[t_1^{(p)}, t_2^{(p)}]\cdot \ldots \cdot [t_{2n-1}^{(p)}, t_{2n}^{(p)}],
$$ 
where the index $n$ does not depend on $p$ (some of the factors may
be trivial).  We denote this product of commutators by $\mcomm{t_i}$.

Using the strong approximation theorem for the spinorial kernel
$\Orth'(L)$ (see \cite[104:4]{OM}) we find $h_i\in \Orth'(L\otimes
\QQ)$ such that
$$
\lVert t_i^{(p)}-h_i\rVert_p<\varepsilon \quad \forall
\ p|m\quad\text{and}\quad
\lVert h_i\rVert_p=1\quad \forall\ p\nmid m.
$$
If $\varepsilon$ is sufficiently small then $h_i\in \Orth'(L)$ and
$\lVert t(e,u)\mcomm{h_i}^{-1}-1\rVert_p$ will be small for any prime
divisor of $m$.  Then $t(e,u)\mcomm{h_i}^{-1}\equiv 1 \mod m$. It
follows that $t(e,u)$ belongs to the commutator subgroup of
$\Orth'(L)=E(L)$.

This completes the proof of Theorem~\ref{com2}.

\subsection{Orthogonal groups over number fields.}\label{subsect:numberfields}

A version of Theorem~\ref{thKn} holds over an algebraic number
field. To formulate this, collecting the remarks in \cite[\S 5]{Kn1},
we must give a suitable extended version of the Kneser conditions. We
say that a lattice $L$ over the ring of integers $\cO_K$ of a number
field $K$ satisfies the Kneser conditions if $L$ is even and
represents $-2$; there exists a real place $\nu$ of $K$ such that the
Witt index of $L\otimes K_\nu$ is at least $2$; and the the
$\pi$-rank $\rank_\pi(L)$ is at least $5$ (respectively at least $6$)
if $\pi$ is a place such that the residue field $k_\pi$ is $\FF_3$
(respectively $k_\pi=\FF_2$).

\begin{theorem}\label{thKn2} (\cite{Kn1})
Suppose $L$ is an integral lattice over $\cal O_K$ satisfying the
Kneser conditions. Then $\Orth'(L)=\SO(L)\cap \ker \sn_K$ is generated
by the products of reflections $\sigma_a\sigma_b$ where $a,b\in L$ and
$a^2=b^2=-2$.
\end{theorem}

In this context we have the following result, analogous to
Theorem~\ref{com1} and Theorem~\ref{com2}.

\begin{theorem}\label{tha1}
Let $L$ be a lattice over the ring of integers $\cO_K$ of an algebraic number
field $K$ which satisfies the Kneser conditions. Then $\abel{\Orth'(L)}$
is an abelian $2$-group. Its order divides $2^{N-1}$, where $N$ is the
number of different $\Orth'(L)$-orbits of $(-2)$-vectors in $L$.

If $L$ contains two hyperbolic planes and $\cO_K$ is a principal ideal
ring then $\abel{\Orth'(L)}$ is trivial.
\end{theorem}
\begin{proof}
The first part of the theorem is similar to Theorem~\ref{com1}.  We
show briefly how to generalise the proof of Theorem~\ref{com2} to the
case of algebraic number fields.  According to \cite{Va2} (see
also \cite{HOM}) the group $\SL_2(\cO_K)$ is generated by the
unipotent matrices $\left(\smallmatrix 1&a\\0&1\endsmallmatrix\right)$
and $\left(\smallmatrix 1&0\\b&1\endsmallmatrix\right)$ where $a,b\in
\cO_K$.

If $\cO_K$ is a principal ideal domain then Lemma~\ref{SO(2,2)} is
still true if we replace $\SO^+(U\oplus U_1)$ by $\Orth'(U\oplus
U_1)$.  The proof is the same: one uses the action~\eqref{isom(2,2)}
and the elementary divisor theorem, which is true for principal ideal
domains in its classical matrix form (there exist $g$, $h\in
\SL_2(\cO_K)$ such that $gMh$ is diagonal). Moreover using \eqref{tSL}
and Vaserstein's result from~\cite{Va2} we obtain $\Orth'(U\oplus
U_1)=E(U\oplus U_1)$.  Using this version of Lemma~\ref{SO(2,2)}, we
see that Proposition~\ref{ET} is still true over a principal ideal
domain.  (There are no changes in the proof.)  Now we can repeat the
proof of Theorem~\ref{com2} using the strong approximation theorem and
the positive solution of the congruence subgroup problem
(see~\cite{Kn2}).
\end{proof}

\section{Fundamental groups}\label{section:fundamental}

In this section we use our results above to compute the fundamental
groups of some locally symmetric varieties and their compactifications.

Let $\cD$ be a bounded symmetric domain and let $\Gamma$ be an
arithmetic group acting on $\cD$. Put $X=\Gamma\backslash \cD$.

\begin{lemma}\label{surjection}
There is a surjective homomorphism $\Gamma\onto \pi_1(X)$, which is an
isomorphism if $\Gamma$ acts freely on $\cD$.
\end{lemma}
\begin{proof}
The map $\phi\colon \Gamma\to \pi_1(\cD/\Gamma)$ is defined as
follows. Choose a base point $p_0\in\cD$, and suppose
$\gamma\in\Gamma$. Since $\cD$ is connected and simply-connected, we
may join $p_0$ and $\gamma(p_0)$ by a path $\sigma_\gamma$ and any two
such paths are homotopic. The quotient map $\pi\colon \cD\to
\Gamma\backslash \cD$ makes this into a loop $\pi\circ\sigma_\gamma$
based at $x_0=\pi(p_0)$, and we define $\phi(\gamma)$ to be the
homotopy class $[\pi\circ\sigma_\gamma]\in \pi_1(X,x_0)$.

However, $\pi_1(X, x_0)$ is isomorphic to $\pi_1(X,x)$ for any base
point $x\in X$. It is easy to check that the map $\phi$ is
well-defined and has the required properties.
\end{proof}

\begin{lemma}\label{fixed}
If $\gamma$ has fixed points in $\cD$ then $\gamma\in\ker\phi$.
\end{lemma}

\begin{proof}
In the proof of Lemma~\ref{surjection} we may choose $p_0$ and
$\sigma_\gamma$ freely, so we choose $p_0$ to be a fixed point of
$\gamma$ and $\sigma_\gamma$ to be the constant path at $p_0$. Then
$[\pi\circ\sigma_\gamma]=1$.
\end{proof}

Now we pass to compactifications of $X$. Let $\Bar X$ denote a normal
compactification of $X$ and let $\Tilde X$ denote a projective smooth
model of $\Bar X$.

\begin{proposition}\label{compactifications}
There are surjections $\Gamma\onto \pi_1(\Tilde X)$ and
$\Gamma\onto\pi_1(\Bar X)$, both factoring through $\phi\colon
\Gamma\to \pi_1(X)$.
\end{proposition}

\begin{proof}
Note first of all that $\pi_1(\Tilde X)$ does not depend on the choice
of the model $\Tilde X$ (see for example~\cite{HK} or~\cite[Lemma
  1.3]{Sa}). So we may take a toroidal compactification $X'$ of $X$
with only finite quotient singularities and $\Tilde X$ a resolution of
$X'$. Since $X\subset X'$ there is a surjection $\pi_1(X)\to
\pi_1(X')$. By~\cite[\S 7]{Kol}, resolving finite quotient
singularities does not change the fundamental group, so we have (by,
for example,~\cite[Lemma 1.2]{Sa}) a surjection $\Gamma\onto
\pi_1(\Tilde X)$ factoring through $\pi$.

For the case of $\Bar X$, in particular for the Satake
compactification, one may, as in~\cite[p.~42]{Sa}, apply the
remark~\cite[p.~56]{Fu} that the inclusion of an open subvariety in a
normal variety induces a surjection on fundamental groups.
\end{proof}

\begin{corollary}\label{zerocondition}
If $\Gamma$ is generated by elements $\gamma\in\Gamma$ with fixed
points in $\cD$, then $\pi_1(X)=\pi_1(\Tilde X)=1$.
\end{corollary}

\begin{proof}
This follows immediately from Lemma~\ref{fixed} and
Proposition~\ref{compactifications}.
\end{proof}
For an integral lattice  $L$ of real signature $(2,n)$ one can determine
the hermitian homogeneous domain of type IV
$$
\cD(L)=\{[Z]\in \PP(L\otimes \CC)\,|\, (Z,Z)=0,\ (Z,\bar Z)>0\}^+
$$
where $+$ means a connected component. In \cite{GHSK3}--\cite{GHSsympl}
we studied the geometry of the modular varieties
$$
\cF(L)=\Tilde\Orth^+(L)\backslash \cD(L)
\quad\text{and}\quad  \cS\cF(L)=\Tilde\SO^+(L)\backslash \cD(L).
$$
For $L=L_{2d}$ (see (\ref{L2d})) the variety $\cF_{2d}=\cF(L_{2d})$
is the moduli space of $\Kthree$ surfaces with a polarisation
of degree~$2d$. The variety  $\cS\cF_{2d}$
corresponds to the addition of a spin structure
(see \cite[\S~5]{GHSK3}).

\begin{theorem}\label{K3simplyconnected}
Let $L$ be a lattice with $\sign_\RR(L)=(2,n)$ satisfying the
condition of Theorem \ref{com2}. Then $\cF(L)$ and $\cS\cF(L)$, as
well as any smooth complete model of $\cF(L)$ or $\cS\cF(L)$, are
simply connected. In particular this is true for the moduli spaces
$\cF_{2d}$ and $\cS\cF_{2d}$.
\end{theorem}

\begin{proof}
In view of Corollary~\ref{zerocondition} it is enough to verify that
$\Tilde\Orth^+(L_{2d})$ and $\Tilde\SO^+(L_{2d})$ are generated by
elements having fixed points in $\cD_{L_{2d}}$.  It is easy to see
that $L_{2d}$ satisfies the Kneser conditions.  So by
Theorem~\ref{thKn} and Corollary~\ref{comparespinors},
$\Tilde\SO^+(L_{2d})=\Orth'(L_{2d})$ is generated by products of pairs
of reflections, and $\Tilde\Orth^+(L_{2d})$ is generated by
reflections. Both reflections and the products of two reflections have
fixed points, so the result follows.
\end{proof}

\begin{proposition}\label{enriquessimplyconnected}
The moduli space $\cE$ of Enriques surfaces, and any smooth compactification
of it, are simply-connected.
\end{proposition}

\begin{proof}
This follows from the hard fact that the moduli space of Enriques surfaces
is rational~\cite{KoEnriques}. However, for simply-connectedness we
can give a quick proof using the results above. The moduli space
$\cE$ is associated with the lattice $L=U(2)\oplus U\oplus E_8(-2)$,
which has $2$-rank $2$ and therefore does not satisfy the Kneser
conditions. But $\cE=\Orth^+(L)\backslash\cD_L$ is also
equal to $\Orth^+(L')\backslash\cD_{L'}$, where $L'=U\oplus
U(2)\oplus E_8(-1)$, since $L$ is obtained from $L'$ as the sublattice
of $L'(2)$ of index $4$ where the generators $e$, $f$ of $U(4)$ are
replaced by $e/2$ and $f/2$: see \cite{KoEnriques}.

Since $L'$ does satisfy the Kneser conditions, Theorem~\ref{com1}
tells us that $\Tilde\Orth^+(L')$ is generated by pairs of
reflections, and these have fixed points. But $\Tilde\Orth^+(L')$ is
of index~$2$ in $\Orth^+(L')$, and the reflection that interchanges
the two generators of $U(2)$ is the extra generator that we need. It
also has fixed points in $\cD_{L'}$, so by Lemma~\ref{zerocondition}
we are done.
\end{proof}

Apart from elements with fixed points there are also other elements
in the kernel of $\Gamma\onto \pi_1(\Bar X)$, namely those coming from
the unipotent radical of parabolic subgroups.
By Lemma~\ref{lemma:unique_extension}, a unimodular transvection
$t(e,v)$ is determined by a unimodular isotropic vector, $e^2=0$,
$\div(e)=1$, and by $v\in e^\perp_L$.  Thus $e$ defines a
zero-dimensional cusp of the modular variety
$X=\Tilde\SO^+(L)\setminus \cD_L$.  In other words $t(e,v)$ is an
element of the corresponding parabolic subgroup $P$, and hence it
belongs to the centre of the unipotent radical $U_P$ of $P$.
Different transvections correspond to different $0$-dimensional cusps.
According \cite[Theorem 1.5]{Sa} and \cite[Corollary 1.6]{Sa}, $E(L)$ is
contained in the kernel of the surjection $\phi\colon \Tilde\SO^+(L)
\to \pi_1 (\Tilde X)$.

For the moduli space $\cA_t$ of abelian surfaces with a polarisation
of type $(1,t)$ the lattice that occurs is $\Lambda_{2t}=2U\oplus\latt{-2t}$ and
the group is the paramodular group $\Gamma_{2t}$. As we have seen,
both the Kneser conditions and the conclusions of Theorem~\ref{com1}
fail in this case. However, the results of this paper together with
those of~\cite{Sa} still give us results about the fundamental
groups.
\begin{theorem}\label{Atsimplyconnected}
Any smooth model $\Tilde\cA_t$ of a compactification of $\cA_t$ is
simply-connected.
\end{theorem}
\begin{proof}
We cannot apply Proposition~\ref{generators} to  the lattice $\Lambda_{2t}$ 
but the last identity~\eqref{otildeplus} of this proposition 
is still true for $\Lambda_{2t}$.
According to \cite{GH2} there exists an isomorphism 
$\Phi\colon\Gamma_t/\{\pm\Eins\}\to \Tilde\SO^+(\Lambda_{2t})$.
For  the paramodular group~$\Gamma_t$ we have 
$$
\Gamma_t=\group{\Gamma_t^J, J_t}
\quad \text{ where }\quad
J_t=\begin{pmatrix}
0&0&-1&  0\\
0&0& 0&-1/t\\
1&0& 0&  0\\
0&t& 0&  0
\end{pmatrix}
$$
and $\Gamma_t^J$ is the Jacobi subgroup of the paramodular group.
This follows from the elementary divisor theorem for the symplectic
group: see, for example,~\cite{Gr1}.

We know (see \cite{GH2}) that $\Phi(\Gamma_t^J)=\Gamma^J(\Lambda_{2t})$, which 
is generated by transvections (see Subsection~\ref{subsect:jacobi_group}).
Then 
$$
\Phi(J_t)=
\begin{pmatrix}
 0& 0& 0&  0&-1\\
 0& 0& 0& -1& 0\\
 0& 0& 1&  0& 0\\
 0&-1& 0&  0& 0\\
-1& 0& 0&  0& 0
\end{pmatrix}=
\sigma_{e+f}\,\sigma_{e_1+f_1}
$$ 
where we use  notations of  Subsection~\ref{subsect:jacobi_group}.
As in the proof of Proposition 1.6 we see that $\Phi(J_t)$ is a transvection.
Therefore 
\begin{equation}\label{last}
\Tilde\SO^+(\Lambda_{2t})=E(\Lambda_{2t}),\qquad
\Tilde\Orth^+(\Lambda_{2t})=\group{\Gamma^J(\Lambda_{2t}), \sigma_{e_1-f_1}}
\end{equation}
\end{proof}
In \cite[Theorem 3.4]{Sa} it was proved that $\Tilde\cA_p$ is
simply-connected for any odd prime $p$. Also in \cite{Sa} one may find
examples of locally symmetric varieties that are not
simply-connected. However, in all these cases one has, in particular,
that the fundamental group is finite and therefore
the irregularity is zero.

In a similar way, by combining the results of~\cite{Sa} and those of
Subsection~\ref{subsect:numberfields}, one can prove that some Shimura
varieties (considered as complex manifolds) are simply-connected.

\bibliographystyle{alpha}

\begin{thebibliography}{AMRT}

\bibitem[AMRT]{AMRT} A.~Ash, D.~Mumford, M.~Rapoport, Y.~Tai, {\it
  Smooth compactification of locally symmetric varieties.} Lie Groups:
  History, Frontiers and Applications, Vol. IV. Math. Sci.  Press,
  Brookline, Mass., 1975.

\bibitem[Bo]{Bo} R.E.~Borcherds, \emph{Automorphic forms with
  singularities on Grassmannians.}  Invent. Math. {\bf 132} (1998),
  491--562.

\bibitem[Br]{Br} E.~Brieskorn, \emph{Die Milnorgitter der
  exzeptionellen unimodularen Singularit\"aten.}  Bonner Mathematische
  Schriften, {\bf 150} (1983).

\bibitem[De]{De} C.~Desreumaux, \emph{Construction de formes
  automorphes r\'eflectives sur un espace de dimension $4$.}
  J. Th\'eorie des Nombres de Bordeaux {\bf 18} (2006), 89--111.

\bibitem[Eb]{Eb} W.~Ebeling, \emph{An arithmetic characterisation of
  the symmetric monodromy groups of singularities.} Invent. Math. {\bf
  77} (1984), 85--99.

\bibitem[Ei]{Ei} M.~Eichler, \emph{Quadratische Formen und orthogonale
  Gruppen.} Grundlehren der mathematischen Wissenschaften {\bf
  63}. Springer-Verlag, Berlin-New York, 1952.

\bibitem[Fu]{Fu} W.~Fulton, \emph{Introduction to toric varieties.}
  Annals of Mathematics Studies {\bf 131}, Princeton University Press,
  Princeton, NJ, 1993.

\bibitem[Gr1]{Gr1} V.~Gritsenko, \emph{Irrationality of the moduli
  spaces of polarized abelian surfaces.} Int. Math. Res. Notices {\bf
  6} (1994), 235--243.

\bibitem[Gr2]{Gr2} V.~Gritsenko, \emph{Modular forms and moduli spaces
  of abelian and $\Kthree$ surfaces.} Algebra i Analiz {\bf 6} (1994),
  65--102; English translation in St. Petersburg Math. J. {\bf 6}
  (1995), 1179--1208.

\bibitem[GH1]{GH1} V.~Gritsenko, K.~Hulek, \emph{Commutator coverings
  of Siegel modular threefolds.}  Duke Math. J {\bf 94} (1998),
  509--542.

\bibitem[GH2]{GH2} V.~Gritsenko, K.~Hulek, \emph{Minimal Siegel
  modular threefolds.}  Math. Proc. Cambridge Philos. Soc. {\bf 123}
  (1998), 461--485.

\bibitem[GHS1]{GHSK3} V.~Gritsenko, K.~Hulek, G.K.~Sankaran, \emph{The
  Kodaira dimension of the moduli of K3 surfaces.} Invent. Math.  {\bf
  169} (2007), 519-567.

\bibitem[GHS2]{GHSorth} V.~Gritsenko, K.~Hulek, G.K.~Sankaran,
  \emph{Hirzebruch-Mumford proportionality and locally symmetric
    varieties of orthogonal type.} Documenta Math. {\bf 13}, 1-19
  (2008).

\bibitem[GHS3]{GHSsympl} V.~Gritsenko, K.~Hulek, G.K.~Sankaran,
  \emph{Moduli spaces of irreducible symplectic manifolds.}  {\tt
    arXiv:0802.2078}, 41 pp.

\bibitem[HO'M]{HOM} A.J.~Hahn, O.T.~O'Meara, \emph{The classical groups
  and $K$-theory.}  Grundlehren der mathematischen Wissenschaften {\bf
  291}. Springer-Verlag, Berlin-New York, 1989.

\bibitem[HK]{HK} H.~Heidrich, F.W.~Kn\"oller, \emph{\"Uber die
  Fundamentalgruppen von Siegelscher Modulvariet\"aten vom Grade~$2$.}
  Manuscr. Math. {\bf 57} (1987), 249--262.

\bibitem[Kn1]{Kn1} M.~Kneser, \emph{Erzeugung ganzzahliger orthogonaler
  Gruppen durch Spiegelungen.} Math. Ann. {\bf 255} (1981),
  453--462.

\bibitem[Kn2]{Kn2} M.~Kneser, \emph{Normalteiler ganzzahliger
  Spingruppen.}  J. reine angew. Math. {\bf 311} (1979), 191--214.

\bibitem[KnS]{KnS} M.~Kneser, \emph{Quadratische Formen.}  Neu
  bearbeitet und herausgegeben in Zusammenarbeit mit Rudolf Scharlau.
  Springer, 2002. [Zbl~1001.11014]

\bibitem[Kol]{Kol} J. Koll\'ar, \emph{Shafarevich maps and plurigenera
  of algebraic varieties.} Invent. Math. {\bf 113} (1993), 117--215.

\bibitem[Ko1]{Ko1} S.~Kondo, \emph{On the Albanese variety of the
  moduli space of polarized $\Kthree$ surfaces.} Invent. math. {\bf
  91} (1988), 587--593.

\bibitem[Ko2]{Ko2} S.~Kondo, \emph{Moduli spaces of K3 surfaces.}
Compositio Math. {\bf 89} (1993), 251--299.

\bibitem[Ko3]{KoEnriques} S.~Kondo, \emph{The rationality of
  the moduli space of Enriques surfaces.}  Compositio Math. {\bf 91}
  (1993), 159-173.

\bibitem[O'M]{OM} O.T.~O'Meara, \emph{Introduction to quadratic forms.}
  Grundlehren der mathematischen Wissenschaften {\bf
    117}. Springer-Verlag, Berlin-New York, 1963.

\bibitem[P-SS]{P-SS} I.~Piatetskii-Shapiro, I.~Shafarevich, \emph{A
  Torelli theorem for algebraic surfaces of type $\Kthree$.}  Izv.
  Akad. Nauk SSSR, Ser. Mat. {\bf 35} (1971), 530-572.  English
  translation in Math. USSR, Izv. {\bf 5} (1972), 547-588.

\bibitem[Sa]{Sa} G.K.~Sankaran, \emph{Fundamental group of locally
  symmetric varieties.}  Manus. Math. {\bf 90} (1996), 39--48.

\bibitem[Va1]{Va1} L.N.~Vaserstein, \emph{Normal subgroups of
  orthogonal groups over commutative rings.}  Amer. J. Math.  {\bf
  110} (1988), 955--973.

\bibitem[Va2]{Va2} L.N.~Vaserstein, \emph{On the group $\SL_2$ over
  Dedekind rings of arithmetic type.}  Math. USSR--Sbornik {\bf 18}
  (1972), 321--332.

\bibitem[Wa]{Wa} C.T.C~Wall, \emph{On the orthogonal groups of
  unimodular quadratic forms II.} J. reine angew. Math. {\bf 213}
  (1963), 122-136.

\end{thebibliography}

\bigskip

\noindent
V.A.~Gritsenko\\
Universit\'e Lille 1\\
Laboratoire Paul Painlev\'e\\
F-59655 Villeneuve d'Ascq, Cedex\\
France\\
{\tt valery.gritsenko@math.univ-lille1.fr}
\bigskip

\noindent
K.~Hulek\\
Institut f\"ur Algebraische Geometrie\\
Leibniz Universit\"at Hannover\\
D-30060 Hannover\\
Germany\\
{\tt hulek@math.uni-hannover.de}
\bigskip

\noindent
G.K.~Sankaran\\
Department of Mathematical Sciences\\
University of Bath\\
Bath BA2 7AY\\
England\\
{\tt gks@maths.bath.ac.uk}
\end{document}